\documentclass{amsart}[12pt]
\usepackage[top=1.0in, bottom=1.0in, left=1.1in, right=1in]{geometry}
\usepackage{amssymb}
\usepackage{stmaryrd}
\usepackage{mathrsfs}
\usepackage{amsthm}
\usepackage{amsmath}
\usepackage{enumerate}
\usepackage[english]{babel}
\usepackage[arrow,matrix,curve]{xy}
\usepackage{hyperref}

\DeclareMathOperator{\gl}{gl}
\DeclareMathOperator{\res}{res}
\DeclareMathOperator{\tr}{tr}

\newcommand{\sm}{\wedge} 
\newcommand{\tensor}{\otimes} 
\newcommand{\iso}{\cong} 
\renewcommand{\to}{\longrightarrow}
\newcommand{\bs}{\backslash} 
\newcommand{\upi}{\underline{\pi}}
\newcommand{\gh}[1]{\llbracket #1\rrbracket}

\newcommand{\mA}{\mathbb A}
\newcommand{\mC}{\mathbb C}
\newcommand{\mF}{\mathbb F}
\newcommand{\mH}{\mathbb H}
\newcommand{\mK}{\mathbb K}
\newcommand{\mR}{\mathbb R}

\newcommand{\mZ}{\mathbb Z}

\newcommand{\bA}{\mathbf A}
\newcommand{\bL}{\mathbf L}
\newcommand{\bMO}{\mathbf{MO}}
\newcommand{\bMU}{\mathbf{MU}}

\newcommand{\xra}{\xrightarrow}

\newcommand{\Fin}{\mathcal F in}

\newcommand{\GH}{{\mathcal{GH}}}
\newcommand{\SH}{{\mathcal{SH}}}
\newcommand{\Ab}{{\mathcal A}b}

\numberwithin{equation}{section}
\newtheorem{thm}[equation]{Theorem}
\newtheorem{prop}[equation]{Proposition}
\newtheorem{cor}[equation]{Corollary}

\theoremstyle{definition}

\newtheorem{rk}[equation]{Remark}
\newtheorem{eg}[equation]{Example}
\newtheorem{con}[equation]{Construction}

\begin{document}

\title[Splittings of global Mackey functors and regularity of equivariant Euler classes]
{Splittings of global Mackey functors\\and regularity of equivariant Euler classes}

\date{\today; 2020 AMS Math.\ Subj.\ Class.: 55N91, 55P91, 55Q91}

\author{Stefan Schwede}
\address{Mathematisches Institut, Universit\"at Bonn, Germany}
\email{schwede@math.uni-bonn.de}

\begin{abstract}
  We establish natural splittings for the values of global Mackey functors
  at orthogonal, unitary and symplectic groups.
  In particular, the restriction homomorphisms between
  the orthogonal, unitary and symplectic groups of adjacent dimensions are naturally split epimorphisms.

  The interest in the splitting comes from equivariant stable homotopy theory.
  The equivariant stable homotopy groups of every global spectrum form
  a global Mackey functor, so the splittings imply that certain long exact homotopy group sequences
  separate into short exact sequences.
  For the real and complex global Thom spectra $\bMO$ and $\bMU$, the
  splittings imply the regularity of various Euler classes related to the tautological
  representations of $O(n)$ and $U(n)$.
\end{abstract}

\maketitle

\section*{Introduction}

The purpose of this paper is to establish a splitting result for
the values of global Mackey functors at orthogonal, unitary and symplectic groups.
As a corollary, we derive regularity properties
of equivariant Euler classes related to the tautological
representations of $O(n)$ and $U(n)$.
For this introduction, I'll concentrate on the unitary case,
where the splitting includes the following statement:\medskip

{\bf Theorem.}
{\em For every global functor $F$ and every $n\geq 1$,
the restriction homomorphism 
\[ \res^{U(n)}_{U(n-1)}\ :\ F(U(n))\ \to\  F(U(n-1))  \]
  is a naturally split epimorphism.}\medskip

The group $F(U(n))$ then naturally splits
as the direct sum of the kernels of the restriction homomorphisms
$\res^{U(k)}_{U(k-1)}:F(U(k))\to F(U(k-1))$ for $k=0,\dots,n$, by induction.
The above theorem is included in Theorem \ref{thm:GF splitting},
where we exhibit a specific natural splitting.
Besides the orthogonal, unitary and symplectic groups, 
the splitting also has an analog for symmetric groups, see Remark \ref{rk:symmetric};
this case is a direct generalization of Dold's arguments
\cite{dold:decomposition} from group cohomology to global functors.
The splittings do {\em not} have analogs for alternating, special orthogonal
and special unitary groups, see Section \ref{sec:nonex}.\smallskip

The interest in the splitting comes from equivariant stable homotopy theory.
Indeed, for every global equivariant spectrum $X$,
i.e., an object of the global stable homotopy category \cite[Section 4]{schwede:global},
the collection of equivariant homotopy groups $\pi_*^G(X)$
naturally forms a $\mZ$-graded global functor as $G$ varies over all compact Lie groups.
Hence the restriction homomorphism
$\res^{U(n)}_{U(n-1)}:\pi_*^{U(n)}(X)\to \pi_*^{U(n-1)}(X)$ is a naturally split epimorphism.
This is a genuinely global phenomenon: as we illustrate in Example \ref{eg:counter example},
this restriction homomorphism is not surjective for general $U(n)$-spectra.

An interesting special case is the global Thom spectrum $\bMU$
defined in \cite[Example 6.1.53]{schwede:global}.
For every compact Lie group $G$, the underlying $G$-homotopy type of $\bMU$
is that of tom Dieck's homotopical equivariant bordism \cite{tomDieck:bordism_integrality}.
This equivariant version of the complex bordism spectrum has been the object of much study,
as it is related to equivariant bordism of stably almost complex manifolds
\cite{loeffler-bordismengruppen, tomDieck:bordism_integrality},
equivariant complex-oriented cohomology theories \cite{cole-greenlees-kriz:universal, greenlees:survey},
and equivariant formal groups
laws \cite{cole-greenlees-kriz:equivariant_formal, greenlees:survey, hausmann:group_law}. 
Our second main result is:\medskip

{\bf Theorem.}
{\em For all $k_1,\dots,k_m\geq 1$, the Euler class of the tautological representation
  of $U(k_1)\times \dots\times U(k_m)$ on $\mC^{k_1+\dots+k_m}$ is a non zero-divisor
  in the graded ring $\bMU^*_{U(k_1)\times \dots\times U(k_m)}$.
}\medskip

This result will be proved in Corollary \ref{cor:block regularity}.
As we explain in Section \ref{sec:regularity},
the regularity property is a relatively direct consequence of the
surjectivity of the restriction homomorphisms
$\res^{U(n)}_{U(n-1)}:\bMU^*_{U(n)}\to \bMU^*_{U(n-1)}$.
For $n=1$ and $n=2$, the surjectivity of the restriction homomorphisms,
and hence the regularity of the Euler classes,
were previously known and have a more direct proof.
Indeed, the standard embeddings $U(0)\to U(1)$ and $U(1)\to U(2)$ admit retractions by continuous
group homomorphisms; inflation along such a retraction then provides a splitting
to the restriction homomorphism.
For $n\geq 3$, however, the embedding $U(n-1)\to U(n)$
does not admit such a retraction, and the splitting only exists after passage
to unreduced suspension spectra of the global classifying spaces.
To the best of the author's knowledge, ours is the first general regularity result for
unitary groups of arbitrary rank.

Many of the major structural results about equivariant homotopical bordism
are so far only known for {\em abelian} compact Lie groups,
and regularity properties of Euler classes play an important role.
Examples of such results include the following:

\begin{itemize}
\item For abelian compact Lie groups, the ring $\bMU^*_G$ is a free module
  over the non-equivariant homotopy ring $\bMU^*$ and concentrated in even degrees
  \cite[Theorem 5.3]{comezana}, \cite{loeffler-equivariant}.
\item The $MU$-cohomology of the non-equivariant classifying space of
  an abelian compact Lie group is the completion of $\bMU^*_G$ at the augmentation ideal
  \cite[Theorem 1.1]{comezana-may}, \cite{loeffler-bordismengruppen}.
\item For abelian compact Lie groups, $\bMU^*_G$ carries the universal
  $G$-equivariant formal group law \cite[Theorem A]{hausmann:group_law}.
\item The collection of rings  $\bMU^*_G$ for all abelian compact Lie groups
  carries the universal global formal group \cite[Theorem C]{hausmann:group_law}.
\end{itemize}
I hope that our regularity results might be useful to understand if and how
the above results generalize from abelian compact Lie groups to general compact Lie groups.

Our splittings for global functors translate directly into stable global splittings
of the global classifying spaces $B_{\gl}O(n)$, $B_{\gl}U(n)$ and $B_{\gl}S p(n)$
of the orthogonal, unitary and symplectic groups, see Corollary \ref{cor:split_B_gl}.
On underlying non-equivariant homotopy types,
the stable splittings of $B U(n)$ and $B S p(n)$ are due to Snaith,
see \cite[Theorem 4.2]{snaith:algebraic_cobordism}, \cite[Theorem 2.2]{snaith:localied_stable},
and the stable splitting of $B O(n)$
was constructed by Mitchell and Priddy \cite[Theorem 4.1]{mitchell-priddy:double_coset}.
If $G$ is a compact Lie group, we can apply the forgetful functor
from the global to the genuine $G$-equivariant stable homotopy category,
compare \cite[Theorem 4.5.24]{schwede:global}.
We obtain $G$-equivariant stable splittings of the classifying $G$-spaces
for $G$-equivariant real, complex and quaternionic vector bundles;
as far as I know, these splittings are new.
\smallskip

{\bf Acknowledgments.} I would like to thank Markus Hausmann, John Greenlees and Peter May
for valuable feedback on this paper.
The research in this paper was supported by the DFG Schwerpunktprogramm 1786 ‘Homotopy Theory
and Algebraic Geometry’ (GZ SCHW 860/1-2) and by the Hausdorff Center for Mathematics at the
University of Bonn (DFG GZ 2047/1, project ID 390685813).

\section{The splitting}

In this section we formulate and prove our first main result, Theorem \ref{thm:GF splitting},
the natural splitting of the values of a global functor at orthogonal, unitary and symplectic groups.
We recall that a {\em global functor} in the sense of \cite[Definition 4.2.2]{schwede:global}
is an additive functor from the global Burnside category of \cite[Construction 4.2.1]{schwede:global}
to the category of abelian groups. 
In more explicit terms, a global functor specifies values on all compact Lie groups,
restriction homomorphisms along continuous group homomorphisms,
and transfers along inclusions of closed subgroups;
this data has to satisfy a short list of explicit relations that can be found after
Theorem 4.2.6 of \cite{schwede:global}.
The data of such a global functor is equivalent to
that of a {\em `functor with regular Mackey
structure'} in the sense of Symonds \cite[\S 3, p.177]{symonds-splitting}.

The proof of our splitting in Theorem \ref{thm:GF splitting}
is inspired by Dold's elegant proof \cite{dold:decomposition}
of Nakaoka's splitting \cite{nakaoka:decomposition}
of the cohomology of symmetric groups.
We generalize Dold's strategy in three ways:
\begin{itemize}
\item from symmetric groups to orthogonal, unitary and symplectic groups,
\item from the non-equivariant to the global context, and
\item from group cohomology to general global functors.
\end{itemize}
The proof of our splittings relies on the full global structure,
as the splitting maps involve restriction, inflation, and transfers.
\smallskip

We let $\mK$ be one of the real division algebras $\mR, \mC$ or $\mH$.
We denote by $G(n)$ the compact Lie group of $(n\times n)$-matrices $A$ over $\mK$
that satisfy $A\cdot \bar A^t=\bar A^t\cdot A=E_n$,
where $\bar A^t$ is the conjugate transpose matrix
(for $\mK=\mR$, conjugation is the identity).
So $G(n)$ is the orthogonal group $O(n)$ for $\mK=\mR$,
it is the unitary group $U(n)$ for $\mK=\mC$,
and it is the symplectic group $S p(n)$ for $\mK=\mH$.

We let $F$ be a global functor.
We write
\[ i_n \ : \ G(n-1)\ \to \ G(n) \ , \ \quad A \ \longmapsto \
  (\begin{smallmatrix}A & 0 \\ 0 & 1 \end{smallmatrix})
\]
for the standard embedding. This continuous monomorphism induces a restriction
operation $i_n^*$, which is a morphism from $G(n)$ to $G(n-1)$
in the global Burnside category. The global functor sends it to
a restriction homomorphism
\[ F(i_n^*) \ : \ F(G(n))\ \to \ F(G(n-1))\ .\]
We write 
\[ \tr_{m,n}\ : \ F(G(m)\times G(n))\ \to \ F(G(m+n)) \]
for the transfer homomorphism associated to the continuous monomorphism
\[  \mu_{m,n}\ :\ G(m)\times G(n)\ \to\ G(m+n) \ , \quad (A,B)
\ \longmapsto \
  (\begin{smallmatrix}A & 0 \\ 0 & B \end{smallmatrix})\ .\]

We recall the double coset formula for the subgroups
$G(n-1)$ and $G(k)\times G(n-k)$ inside $G(n)$.
The following result ought to be well-known to experts;
in the unitary situation, the second summand in the following
double coset formula actually vanishes, and similar double coset formulas
were established in \cite[Example IV.9]{feshbach} and \cite[Lemma 4.2]{symonds-splitting}.

\begin{prop}\label{prop:double coset U}
  Let $\mK$ be one of the real division algebras $\mR, \mC$ or $\mH$,
  and let $F$ be a global functor.
  Then for every $1\leq k\leq n-1$, the relation
  \begin{align*}
 F(i_n^*)\circ \ &\tr_{k,n-k}\ = \\  
  &\tr_{k,n-k-1}\circ F( (G(k)\times i_{n-k})^*)\
  -\   \tr_{\Delta}\circ F( (i_k\times i_{n-k})^*)\
  + \ \tr_{k-1,n-k}\circ F( (i_k\times G(n-k))^*)
  \end{align*}
  holds as homomorphisms $F(G(k)\times G(n-k))\to F(G(n-1))$, where
  $\tr_\Delta$ denotes the transfer along the closed embedding
  \[ \Delta \ : \ G(k-1)\times G(n-k-1)\ \to \ G(n-1)\ , \quad
(A,B)\ \longmapsto \  \left(\begin{matrix} 
    A & 0 & 0 \\ 
    0 & 1 & 0\\ 
    0 & 0 & B
  \end{matrix}\right)\ .\]
\end{prop}
\begin{proof}
  We write $G(n-1,\sharp)$ for the image of the embedding
  $i_n:G(n-1)\to G(n)$, and we write $G(k,n-k)$
  for the closed subgroup of those block matrices of $G(n)$
  of the form $(\begin{smallmatrix}A & 0 \\ 0 & B \end{smallmatrix})$
  for $(A,B)\in G(k)\times G(n-k)$.
  The double coset space $G(n-1,\sharp)\bs G(n)/G(k,n-k)$
  is a closed interval; more precisely, the map
    \[ [0,\pi]\ \to \ G(n-1,\sharp) \bs G(n) / G(k,n-k) \ ,\quad
    t \ \mapsto\ G(n-1,\sharp)\cdot \gamma(t) \cdot G(k,n-k)   \]
with
  \[
  \gamma(t)\ = \  
    \begin{pmatrix} 
      E_{k-1}  & 0 & 0 & 0 \\
      0 & \cos(t) & 0 & -\sin(t)\\ 
      0 & 0 & E_{n-k-1} & 0\\ 
      0 & \sin(t) & 0 & \cos(t)\end{pmatrix}
  \]
  is a homeomorphism.

  For $t=0$, the stabilizer of the right coset $\gamma(0)\cdot G(k,n-k)$
  under the left $G(n-1,\sharp)$-action is the subgroup $G(k,n-k-1,\sharp)$
  consisting of all matrices of the form
  \[  \left(\begin{matrix} 
        A & 0 & 0 \\ 
        0 & B & 0\\ 
        0 & 0 & 1
      \end{matrix}\right) \]
  with $(A,B)\in G(k)\times G(n-k-1)$.
  For $t=\pi$, the left stabilizer of the right coset $\gamma(\pi)\cdot G(k,n-k)$
  is the subgroup $G(k-1,\sharp,n-k)$
  consisting of all matrices of the form
  \[  \left(\begin{matrix} 
        A & 0 & 0 \\ 
        0 & 1 & 0\\ 
        0 & 0 & B
      \end{matrix}\right) \]
  with $(A,B)\in G(k-1)\times G(n-k)$.
  For $t\in (0,\pi)$,
  the left stabilizer of the right coset $\gamma(t)\cdot G(k,n-k)$
  is the subgroup
  \[  G(k-1,\sharp,n-k-1,\sharp)\ = \ G(k,n-k-1,\sharp)\cap G(k-1,\sharp,n-k)\ . \]
  This shows that the orbit type decomposition is as $ \{0\}\cup (0,\pi)\cup \{\pi\}$.
  
  The double coset formula \cite[IV \S 6]{lms}, \cite[Theorem 3.4.9]{schwede:global}
  thus has three summands. The first summand indexed by $\gamma(0)=E_n$ contributes
  \[ \tr^{G(n-1,\sharp)}_{G(k,n-k-1,\sharp)}\circ\res^{G(k,n-k)}_{G(k,n-k-1,\sharp)}\ . \]
  Under the identification of $G(k)\times G(n-k-1)$ with $G(k,n-k-1,\sharp)$,
  this becomes the term
  \[ \tr_{k,n-k-1}\circ F( (G(k)\times i_{n-k})^*)\ . \]
  The summand indexed by $(0,\pi)$ occurs with coefficient $-1$,
  the internal Euler characteristic of the open interval.
  For $t\in (0,\pi)$, the matrix $\gamma(t)$ centralizes the subgroup $G(k-1,\sharp,n-k-1,\sharp)$;
  so in every global functor, the corresponding conjugation homomorphism
  $c_{\gamma(t)}^*$ is the identity.  
  The second contribution to the double coset formula is thus
  \[   -\   \tr_{G(k-1,\sharp,n-k-1,\sharp)}^{G(n-1,\sharp)}\circ\res^{G(k,n-k)}_{G(k-1,\sharp,n-k-1,\sharp)}\ . \]
  In the notation of the theorem, this becomes the term
  $-\tr_\Delta\circ F( (i_k\times i_{n-k})^*)$.
  The third summand indexed by $\gamma(\pi)$ contributes
  \[ \tr^{G(n-1,\sharp)}_{G(k-1,n-k,\sharp)}\circ c_{\gamma(\pi)}^*\circ\res^{G(k,n-k)}_{G(k-1,\sharp,n-k)}\ . \]
  The following square of group homomorphisms commutes:
  \[ \xymatrix{
      G(k-1)\times G(n-k)\ar[rr]^-{i_k\times G(n-k)}\ar[d]_{\mu_{k-1,n-k}} &&
      G(k)\times G(n-k)\ar[d]^{\mu_{k,n-k}} \\
      G(n-1)\ar[r]_-{i_n^*} & G(n)\ar[r]_{c_{\gamma(\pi)}}& G(n)
    } \]
  So under the identification of $G(k-1)\times G(n-k)$ with $G(k-1,n-k,\sharp)$,
  the third summand becomes the term $\tr_{k-1,n-k}\circ F( (i_k\times G(n-k))^*)$.
\end{proof}

\begin{rk}
  In the unitary and symplectic case, i.e., when the skew field is $\mC$ or $\mH$,
  the transfer $\tr_\Delta:F(G(k-1)\times G(n-k-1))\to F(G(n-1))$ that occurs
  in the double coset formula of Proposition \ref{prop:double coset U} is actually zero.
  Indeed, the image of the embedding $\Delta:G(k-1)\times G(n-k-1)\to G(n-1)$
  is centralized by the subgroup of matrices of the form
  \[
    \left(\begin{matrix} 
        E_{k-1} & 0 & 0 \\ 
        0 & \lambda & 0\\ 
        0 & 0 & E_{n-k-1}
      \end{matrix}\right)\ ,\]
  a group isomorphic to $G(1)$.
  Since the groups $U(1)$ and $Sp(1)$ have positive dimension,
  the Weyl group of the image of $\Delta$  has positive dimension.
  The transfer $\tr_{\Delta}$ is thus trivial.
  In the case of the orthogonal groups, the second summand in the double coset formula
  of Proposition \ref{prop:double coset U} is generically non-zero.
\end{rk}

\begin{con}
  As before we let $F$ be a global functor in the sense of \cite[Definition 4.2.2]{schwede:global}.  
  We now formulate the splittings of $F(O(n))$, $F(U(n))$ and $F(S p(n))$.
  We write
  \begin{align*}
    F(O;k)\ &= \ \ker( F(i_k^*):F(O(k))\to F(O(k-1)) )\ ,  \\
    F(U;k)\ &= \ \ker( F(i_k^*):F(U(k))\to F(U(k-1)) )  \text{\quad and}\\
    F(S p;k)\ &= \ \ker( F(i_k^*):F(S p(k))\to F(S p(k-1)) ) 
  \end{align*}
  for the kernels of the restriction homomorphism along $i_k$.
  For $k=0$, we interpret this as $F(O;0)=F(U;0)=F(S p;0)=F(e)$,
  the value of $F$ at the trivial group.
  For $0\leq k\leq n$, we write 
  \[ p_{k,n-k}^*\ : \ F(O(k))\ \to \ F(O(k)\times O(n-k))
  \]
for the inflation homomorphism associated to the projection to the first factor.
We define a natural homomorphism 
\[ \psi_{k,n}\ : \ F(O;k)\ \to \ F(O(n))    \]
as the following composite  
\[ F(O;k)\ \xra{\text{inclusion}} \ F(O(k)) \ \xra{\ p_{k,n-k}^*\ }\
  F(O(k)\times O(n-k)) \ \xra{\tr_{k,n-k}}\  F(O(n))   \ ,\]
and similarly for
\[ \psi_{k,n}\ : \ F(U;k)\ \to \ F(U(n))   \text{\qquad and\qquad}
   \psi_{k,n}\ : \ F(S p;k)\ \to \ F(S p(n))   \ . \]
Because the group $O(0)$ is trivial, the map $\psi_{0,n}$ specializes to
inflation along the unique homomorphism $O(n)\to O(0)$, and
$\psi_{n,n}$ is the inclusion $F(O;n)\to F(O(n))$.
\end{con}

\begin{thm}\label{thm:GF splitting}
  For every global functor $F$, and every $n\geq 1$, the maps
  \begin{align*}
    \sum \psi_{k,n}&\colon  \bigoplus_{k=0}^n \, F(O;k)\ \to \  F(O(n))\ ,\\
    \sum \psi_{k,n} &\colon  \bigoplus_{k=0}^n \, F(U;k)\ \to \  F(U(n))\text{\quad and}\\
    \sum \psi_{k,n}&\colon \bigoplus_{k=0}^n \, F(S p;k)\ \to \  F(S p(n))
  \end{align*}
  are isomorphisms of abelian groups, and
  the restriction homomorphisms
  \begin{align*}
   F(i_n^*) &\colon F(O(n))\to F(O(n-1)) \ , \\
    F(i_n^*) &\colon F(U(n))\to F(U(n-1))\text{\qquad and}\\
    F(i_n^*) &\colon F(S p(n))\to F(S p(n-1))
  \end{align*}
  are naturally split epimorphism.
\end{thm}
\begin{proof}
  We give the argument in the orthogonal case;
  the unitary and symplectic cases are analogous.
  We suppose that $1\leq k\leq n-1$. 
  We precompose the double coset formula of Proposition \ref{prop:double coset U}
  with the homomorphism $p_{k,n-k}^*\circ\text{incl}:F(O;k)\to F(O(k)\times O(n-k))$.
  We observe that the last two of the three summands on the right hand side compose trivially.
  Indeed,
  \[ \tr_{\Delta}\circ F( (i_k\times i_{n-k})^*)\circ p_{k,n-k}^*\circ\text{incl}\
   = \     \tr_{\Delta}\circ p_{k-1,n-k-1}^*\circ F(i_k^*)\circ  \text{incl}\ = \ 0 \ , \]
 and
 \[ \tr_{k-1,n-k}\circ F( (i_k\times O(n-k))^*)\circ p_{k,n-k}^*\circ\text{incl}\
   = \ \tr_{k-1,n-k}\circ p_{k-1,n-k}^*\circ F(i_k^*)\circ   \text{incl}\ = \ 0 \ . \]
 So the double coset formula implies the relation
  \begin{align*}
        F(i_n^*)\circ\psi_{k,n}\
    &= \ F(i_n^*)\circ\tr_{k,n-k}\circ p_{k,n-k}^*\circ\text{incl} \\
  &= \ \tr_{k,n-k-1}
  \circ F( (O(k)\times i_{n-k})^*)\circ p_{k,n-k}^*\circ\text{incl}\\
    &= \ \tr_{k,n-k-1}\circ p_{k,n-k-1}^*\circ\text{incl} \ =\ \ \psi_{k,n-1}\ .
  \end{align*}
  This final relation also holds for $k=0$ by direct inspection, i.e.,
  $F(i_n^*)\circ\psi_{0,n}=\psi_{0,n-1}$.
  
  Now we can prove the claim by induction on $n$.
  The induction starts with $n=0$, where there is nothing to show.
  For $n\geq 1$ we obtain a commutative diagram
  \[ \xymatrix@C=15mm{
      0 \ar[r] & F(O;n) \ar[r]^-{\text{incl}} \ar@{=}[d] &
      \bigoplus_{k=0}^n F(O;k)\ar[r]^-{\text{proj}}\ar[d]^{\sum \psi_{k,n}} &
      \bigoplus_{k=0}^{n-1} F(O;k)\ar[r]\ar[d]_-\iso^{\sum \psi_{k,n-1}} & 0\\
   0 \ar[r] & F(O;n) \ar[r]_-{\text{incl}}&  F(O(n))\ar[r]_-{F(i_n^*)} & F(O(n-1))\ar[r] & 0
    } \]
  The upper row is exact by definition, and the lower row is exact at
  $F(O;n)$ and $F(O(n))$, also by definition. The right vertical map is an isomorphism
  by the inductive hypothesis; so the restriction map $F(i_n^*):F(O(n))\to F(O(n-1))$
  is in fact surjective, and the lower row is also exact.
  Since both rows are exact and the right vertical map is an isomorphism,
  the middle map is an isomorphism, too,
  and $F(i_n^*)$ is a naturally split epimorphism.
\end{proof}

\begin{rk}[Splitting for symmetric groups]\label{rk:symmetric}
  The splittings of Theorem \ref{thm:GF splitting} have an analog for symmetric groups as well.
  This case is a direct generalization of Dold's arguments
  \cite{dold:decomposition} from group cohomology to global functors.
  The argument for symmetric groups is significantly simpler
  than for orthogonal, unitary and symplectic groups
  because the analog of the double coset formula in Proposition \ref{prop:double coset U}
  is easier to derive. Indeed, the relevant double coset space
  $\Sigma_{n-1}\backslash \Sigma_n /\Sigma_{k,n-k}$ is discrete with two points,
  there is no need for an analysis of the orbit type stratification,
  and the relevant double coset formula is 
\[  \res^{\Sigma_n}_{\Sigma_{n-1}}\circ \ \tr_{\Sigma_{k,n-k}}^{\Sigma_n}\ = \  
  \tr_{\Sigma_{k,n-k-1}}^{\Sigma_{n-1}}\circ \res^{\Sigma_{k,n-k}}_{\Sigma_{k,n-k-1}}\
    + \ \tr_{\Sigma_{k-1,n-k}}^{\Sigma_{n-1}}\circ c_{(k, n)}^*\circ \res^{\Sigma_{k,n-k}}_{ (\Sigma_{k-1,n-k})^{(k, n)}}\ . \]
  Here $\Sigma_n$ is the $n$-th symmetric group,
  and we write $\Sigma_{k,n-k}$ for its subgroup 
  consisting of those permutations that leave the subsets $\{1,\dots,k\}$
  and $\{k+1,\dots,n\}$ invariant.
  We abuse notation by identifying $\Sigma_{n-1}$
  with the subgroup of $\Sigma_n$ of permutations that fix the last element $n$;
  finally $(k, n)$ is the transposition that interchanges $k$ and $n$.
  With this double coset formula at hand, the same argument as in
  the proof of Theorem \ref{thm:GF splitting} shows that for every global functor $F$
  and every $n\geq 1$, the analogously defined  map
  \begin{equation} \label{eq:split_symmetric}
  \sum \psi_{k,n}\ : \ \bigoplus_{k=0}^n \, F(\Sigma;k)\ \to \  F(\Sigma_n)   
  \end{equation}
  is an isomorphism,
  and the restriction homomorphism $F(i_n^*):F(\Sigma_n)\to F(\Sigma_{n-1})$
  is a naturally split epimorphism.
  
  In the case of symmetric groups, it suffices to consider a {\em $\Fin$-global functor},
  i.e., the analog of a global functor defined only on finite groups.
  These $\Fin$-global functors have been studied
  under different names in the algebraic literature,
  for example as  `inflation functors' in \cite[p.271]{webb},
  or as `global $(\emptyset,\infty)$-Mackey functors' in \cite{lewis-projective not flat}. 
  The $\Fin$-global functors are a special case of the more general
  class of `biset functors' \cite{bouc:biset}.
  We refer the reader to \cite[Remark 4.2.16]{schwede:global}
  for more details on the comparison.
  I would not be surprised if the splitting \eqref{eq:split_symmetric}
  was already known, possibly in different language, and published somewhere
  in the algebraic literature on the subject. However, I am not aware of a reference.
\end{rk}

Equivariant stable homotopy theory provides examples of global functors.
Indeed, for every global equivariant spectrum $X$,
i.e., an object of the global stable homotopy category \cite[Section 4]{schwede:global},
and every integer $m$, the collection of $m$-th equivariant stable homotopy groups $\pi_m^G(X)$
naturally forms a global functor as $G$ varies over all compact Lie groups.
Moreover, the preferred t-structure on the global stable homotopy category
shows that every global functor arises in this way, see \cite[Theorem 4.4.9]{schwede:global}.
The splittings of Theorem \ref{thm:GF splitting} show that for every global equivariant spectrum $X$,
the restriction homomorphism
$\res^{O(n)}_{O(n-1)}:\pi_*^{O(n)}(X)\to\pi_*^{O(n-1)}(X)$
is a naturally split epimorphism, and for every $0\leq k\leq n$
the graded abelian group $\pi_*^{O(k)}(X)$ is a natural direct summand
of $\pi_*^{O(n)}(X)$. And the analogous statements hold for unitary and symplectic groups.

\begin{eg}\label{eg:counter example}
  The surjectivity of the restriction homomorphism
  $\res^{O(n)}_{O(n-1)}:\pi_*^{O(n)}(X)\to\pi_*^{O(n-1)}(X)$
  is a special feature of global stable homotopy types,
  and it does not hold for general genuine $O(n)$-spectra.
  An easy example for $O(1)\iso \Sigma_2$ is given by the Eilenberg-MacLane spectrum $H M$
  for the $\Sigma_2$-Mackey functor $M$ with $M(\Sigma_2/e)=\mZ/2$ and $M(\Sigma_2/\Sigma_2)=0$.
  The $\Sigma_2$-equivariant stable homotopy groups of $H M$ vanish,
  while the 0-th non-equivariant stable homotopy group of $H M$ is non-trivial.
  In particular, restriction from $\Sigma_2$ to $\Sigma_1$,
  or from $O(1)$ to $O(0)$, is not surjective.
  
  A different kind of example for unitary groups
  is the unreduced suspension spectrum of the free and transitive $U(1)$-space.
  The Wirthm{\"u}ller isomorphism
  shows that the group  $\pi_0^{U(1)}(\Sigma^\infty_+ U(1))$ vanishes.
  The group $\pi_0^e(\Sigma^\infty_+ U(1))$ is isomorphic to $\mZ$,
  so restriction from $U(1)$ to $U(0)$ is not surjective on 0-th equivariant homotopy groups.
\end{eg}

\section{Stable splittings of global classifying spaces}

Snaith has shown that the unreduced suspension spectra of the classifying spaces
$B U(n)$ and $B S p(n)$ stably split into wedges of certain Thom spaces,
see \cite[Theorem 4.2]{snaith:algebraic_cobordism} and \cite[Theorem 2.2]{snaith:localied_stable}.
Mitchell and Priddy obtained such splittings by a different method in
\cite[Theorem 4.1]{mitchell-priddy:double_coset}, and their proof also applies
to stably split the classifying spaces $B O(n)$ and $B\Sigma_n$ of the orthogonal
and the symmetric groups.
Corollary \ref{cor:split_B_gl} below is a global refinement of this splitting,
referring to the unreduced suspension spectra of
the global classifying spaces $B_{\gl}O(n)$, $B_{\gl}U(n)$ and $B_{\gl}S p(n)$
as defined in \cite[Definition 1.1.27]{schwede:global}.
The splitting takes place in the global stable homotopy category $\GH$,
i.e., the localization of the category of orthogonal spectra at the class of
global equivalences \cite[Definition 4.1.3]{schwede:global}.
The global stable homotopy category is a compactly generated tensor triangulated
category, see \cite[Section 4.4]{schwede:global}.

\begin{con}[Global classifying spaces]
In the model of \cite{schwede:global}, unstable global homotopy types are represented
by orthogonal spaces.
Orthogonal spaces are continuous functors to spaces from the category $\bL$
of finite-dimensional inner product spaces and linear isometric embeddings,
compare \cite[Definition 1.1.1]{schwede:global}.
The category $\bL$ is also denoted $\mathscr I$ or $\mathcal I$ by other authors,
and orthogonal spaces are also known as $\mathscr I$-functors,
$\mathscr I$-spaces or $\mathcal I$-spaces.

An important example is the {\em global classifying space}
of a compact Lie group $G$, see \cite[Definition 1.1.27]{schwede:global}.
The construction involves a choice of faithful $G$-representation $V$,
and then
\[ B_{\gl}G \ = \ \bL(V,-)/G \]
is the orthogonal $G$-orbit space of the represented orthogonal space.
The unstable global homotopy type of $B_{\gl}G$ is independent of the choice of faithful
representation, and $B_{\gl}G$ `globally represents' principal $G$-bundles over equivariant
spaces, see \cite[Proposition 1.1.30]{schwede:global}.
In particular, the underlying non-equivariant homotopy type of $B_{\gl}G$ is
a classifying space for the Lie group $G$.

Every orthogonal space has an unreduced suspension spectrum,
compare \cite[Construction 4.1.7]{schwede:global}.
The suspension spectrum of $B_{\gl}G$ comes with a preferred $G$-equivariant homotopy class
\[ e_G \ \in \ \pi_0^G(\Sigma^\infty_+ B_{\gl}G) \ ,\]
the {\em stable tautological class}, defined in  \cite[(4.1.12)]{schwede:global}.
By \cite[Theorem 4.4.3]{schwede:global}, the pair
$(\Sigma^\infty_+ B_{\gl}G, e_G)$ represents the functor $\pi_0^G:\GH\to \Ab$.
\end{con}

\begin{prop}\label{prop:shifted splitting}
  Let $i:L\to K$ be a continuous homomorphism between compact Lie groups
  such that for every global functor $F$, the restriction homomorphism
  $F(i^*):F(K)\to F(L)$ is surjective.
  \begin{enumerate}[\em (i)]
  \item The morphism $i^*:K\to L$ has a section in the global Burnside category.
  \item The morphism $\Sigma^\infty_+ i:\Sigma^\infty_+ B_{\gl}L\to\Sigma^\infty_+ B_{\gl}K$
    has a retraction in the global stable homotopy category.
  \item For every compact Lie group $G$ and every global functor $F$,
    the restriction homomorphism $F((G\times i)^*):F(G\times K)\to F(G\times L)$
    is a naturally split epimorphism.
  \end{enumerate}
\end{prop}
\begin{proof}
  (i) Global functors are defined as additive functors
  from the global Burnside category $\bA$ to the category of abelian groups.
  For the represented global functor $\bA(L,-)$, the hypothesis
  shows that the restriction homomorphism
  \[ i^*\circ -\ : \ \bA(L,K) \ \to\  \bA(L,L) \]
  is surjective.
  Any preimage of the identity is a section to $i^*$.

  (ii) We let $\sigma\in\bA(L,K)$ be a section to $i^*$ as provided by part (i).
  The representability property of the pair $(\Sigma^\infty_+ B_{\gl}K, e_K)$
  provides a unique morphism $\rho:\Sigma^\infty_+ B_{\gl}K\to\Sigma^\infty_+ B_{\gl}L$ in $\GH$
  such that
  \[ \pi_0^K(\rho)(e_K)\ = \ \sigma(e_L) \]
  in the group $\pi_0^K(\Sigma^\infty_+ B_{\gl}L)$. Then
  \[ \pi_0^L(\rho\circ \Sigma^\infty_+ B_{\gl}i)(e_L) \ = \
    \pi_0^L(\rho)(i^*(e_K)) \ = \ i^*(\pi_0^K(\rho)(e_K)) \ = \
    i^*(\sigma(e_L)) \ = \ e_L\ .\]
  The representability property of the pair $(\Sigma^\infty_+ B_{\gl}L, e_L)$
  thus shows that $\rho\circ \Sigma^\infty_+ B_{\gl}i$ is the identity of $\Sigma^\infty_+ B_{\gl}L$.
  So $\rho$ is the desired retraction.
  
  (iii) The global Burnside category admits a biadditive symmetric monoidal structure
  that is given by the product of Lie groups on objects, see \cite[Theorem 4.2.15]{schwede:global}.
  Moreover, $G\times i^*=(G\times i)^*$ by (4.2.14) of \cite{schwede:global}.
  Part (i) provides a section $\sigma:L\to K$ to $i^*$ in the global Burnside category $\bA$.
  So the morphism
  \[  G\times \sigma \ : \ G\times L\ \to \ G\times K\]
  is a section to $(G\times i)^*$.
  Hence for every global functor $F$, the homomorphism $F(G\times \sigma)$
  is a section to $F( (G\times i)^*)$.
\end{proof}

Theorem \ref{thm:GF splitting} and Proposition \ref{prop:shifted splitting} together
show that the global classifying space of $O(n-1)$ is globally-stably
a direct summand of the global classifying space of $O(n)$,
and similarly for the unitary and symplectic groups.
The next corollary refines this splitting and also identifies the summands
as the suspension spectra of a global Thom spaces.

\begin{con}[Global Thom spaces]
  We let $\nu_{n,\mR}$ denote the tautological real $O(n)$-representation on $\mR^n$.
  For every $n\ge 0$, we define a based orthogonal space $M(n,\mR)$ by
  \[ M(n,\mR)(V)\ = \ \bL(\nu_{n,\mR},V)_+\sm_{O(n)} S^{\nu_{n,\mR}}\ . \]
  So $M(n,\mR)$ is the global Thom space over $B_{\gl}O(n)$ of the
  global vector bundle associated to the tautological real $O(n)$-representation.
  We will show in Corollary \ref{cor:represent kernel}
  that the suspension spectrum of $M(n,\mR)$
  represents the kernel of the restriction homomorphism from $O(n)$ to $O(n-1)$.

  The inclusion $S^0\to S^{\nu_{n,\mR}}$ of the $O(n)$-fixed points induces a morphism
  of based orthogonal spaces
  \[  j \ : \  B_{\gl}O(n)_+ = \bL(\nu_{n,\mR},-)_+\sm_{O(n)} S^0\ \to\ 
    \bL(\nu_{n,\mR},-)_+\sm_{O(n)} S^{\nu_{n,\mR}}\ = \ M(n,\mR)\ . \]
  We pass to suspension spectra to obtain a morphism of orthogonal spectra
  \[  \Sigma^\infty j \ : \ \Sigma^\infty_+ B_{\gl}O(n) \ \to\ \Sigma^\infty M(n,\mR) \ .\]
  We write
  \[ w_{n,\mR}\ = \ \pi_0^{O(n)}(\Sigma^\infty j)(e_{O(n)}) \ \in \ \pi_0^{O(n)}(\Sigma^\infty M(n,\mR))\]
  for the image of the stable tautological class
  $e_{O(n)}\in \pi_0^{O(n)}(\Sigma^\infty_+ B_{\gl}O(n))$.
  
  Similarly, we define based orthogonal spaces $M(n,\mC)$ and $M(n,\mH)$ by
  \begin{align*}
    M(n,\mC)(V)\ &= \ \bL^\mC(\nu_{n,\mC},\mC\tensor_{\mR}V)_+\sm_{U(n)} S^{\nu_{n,\mC}}\text{\quad and}\\  
    M(n,\mH)(V)\ &= \ \bL^\mH(\nu_{n,\mH},\mH\tensor_{\mR}V)_+\sm_{S p(n)} S^{\nu_{n,\mH}}\ .
  \end{align*}
  Here $\nu_{n,\mC}$ is the tautological complex $U(n)$-representation on $\mC^n$,
  and  $\nu_{n,\mH}$ is the tautological quaternionic $S p(n)$-representation on $\mH^n$,
  and $\bL^{\mC}(-,-)$ and $\bL^{\mH}(-,-)$ denote the spaces of $\mC$-linear
  and $\mH$-linear isometric embeddings, respectively.
  The analogous construction as for $M(n,\mR)$ provides us with classes
  \[ w_{n,\mC} \ \in \ \pi_0^{U(n)}(\Sigma^\infty M(n,\mC))\text{\qquad and\qquad}
    w_{n,\mH}\ \in \ \pi_0^{S p(n)}(\Sigma^\infty M(n,\mH))\ .\]
\end{con}

\begin{cor}\label{cor:represent kernel}
  The pair $(\Sigma^\infty M(n,\mR), w_{n,\mR})$ represents the functor
  \[ \ker\big(\res^{O(n)}_{O(n-1)}\colon  \pi_0^{O(n)}\to \pi_0^{O(n-1)}\big) \ :
    \GH \ \to \ \Ab \ .\]
  The pair $(\Sigma^\infty M(n,\mC), w_{n,\mC})$ represents the functor
  \[ \ker\big(\res^{U(n)}_{U(n-1)}\colon  \pi_0^{U(n)}\to \pi_0^{U(n-1)}\big) \ :
    \GH \ \to \ \Ab \ .\]
  The pair $(\Sigma^\infty M(n,\mH), w_{n,\mH})$ represents the functor
  \[ \ker\big(\res^{S p(n)}_{S p(n-1)}\colon  \pi_0^{S p(n)}\to \pi_0^{S p(n-1)}\big) \ :
    \GH \ \to \ \Ab \ .\]
\end{cor}
\begin{proof}
  We give the argument in the orthogonal case; the unitary and symplectic cases are analogous.
  We apply the functor $\Sigma^\infty \bL(\nu_{n,\mR},-)_+\sm_{O(n)}-$ from the category
  of based $O(n)$-spaces to the category of orthogonal spectra to the
  mapping cone sequence
  \[ O(n)/O(n-1)_+ \ \to \ S^0 \ \to \ S^{\nu_{n,\mR}}\ \to \ O(n)/O(n-1)_+\sm S^1 \ .\]
  The result is a mapping cone sequence of orthogonal spectra
  \[ \Sigma^\infty_+ B_{\gl}O(n-1) \ \xra{\Sigma^\infty_+ B_{\gl}i_n} \
    \Sigma^\infty_+ B_{\gl} O(n) \ \xra{\Sigma^\infty j \ }\ 
    \Sigma^\infty M(n,\mR) \ \to \ \Sigma^\infty_+ B_{\gl}O(n-1) \sm S^1\ .\]
  Mapping cone sequences of orthogonal spectra define distinguished triangles
  in the global stable homotopy category; taking morphism groups $\gh{-,E}$
  in $\GH$ to an orthogonal spectrum $E$ turns
  the distinguished triangle into a long exact sequence of abelian groups.
  The orthogonal spectra
  $\Sigma^\infty_+ B_{\gl}O(n-1)$ and
  $\Sigma^\infty_+ B_{\gl}O(n)$ represent the functors $\pi_0^{O(n-1)}$
  and $\pi_0^{O(n)}$, respectively, by \cite[Theorem 4.4.3]{schwede:global},
  and the morphism $\Sigma^\infty_+ B_{\gl}i_n$ represents the
  restriction homomorphism $\res^{O(n)}_{O(n-1)}:\pi_0^{O(n)}(E)\to\pi_0^{O(n-1)}(E)$.
  Since the equivariant homotopy groups of $E$ are part of a global functor, the restriction
  homomorphism is surjective by Theorem \ref{thm:GF splitting}.
  So the long exact sequence decomposes into short exact sequences
  \[ 0 \ \to \ \gh{\Sigma^\infty M(n,\mR),E} \ \xra{f\mapsto f_*(w_{n,\mR})} \ \pi_0^{O(n)}(E)\ \xra{\res^{O(n)}_{O(n-1)}}\
    \pi_0^{O(n-1)}(E)\ \to \ 0 \ .\]
  This proves the claim.
\end{proof}

For the homotopy group global functor $\upi_0(\Sigma^\infty_+ B_{\gl}O(n))$,
Theorem \ref{thm:GF splitting} specializes to a splitting 
\begin{align*}
       \sum \psi_{k,n}\colon  \bigoplus_{k=0}^n \,
  \ker\big(\res^{O(k)}_{O(k-1)}:\pi_0^{O(k)}(\Sigma^\infty_+ B_{\gl}O(n))\to \pi_0^{O(k-1)}&(\Sigma^\infty_+ B_{\gl}O(n))\big)\\
  &\xra{\ \iso \ } \  \pi_0^{O(n)}(\Sigma^\infty_+ B_{\gl}O(n))\ .
\end{align*}
So there is a unique collection of classes $s_k\in \pi_0^{O(k)}(\Sigma^\infty_+ B_{\gl}O(n))$
such that $\res^{O(k)}_{O(k-1)}(s_k)=0$ for all $k=0,\dots,n$
and
\[ \sum_{k=0}^n \psi_{k,n}(s_k)\ =  \ e_{O(n)}\ . \]
Corollary \ref{cor:represent kernel} provides
a unique morphism $\Psi_{k,n}:\Sigma^\infty M(k,\mR)\to \Sigma^\infty_+ B_{\gl}O(n)$
in the global stable homotopy category such that
\[ \pi_0^{O(k)}(\Psi_{k,n})(w_{k,\mR}) \ = \ s_k \ . \]
The analogous unitary and symplectic arguments provide morphisms
$\Psi_{k,n}:\Sigma^\infty M(k,\mC)\to \Sigma^\infty_+ B_{\gl}U(n)$
and $\Psi_{k,n}:\Sigma^\infty M(k,\mH)\to \Sigma^\infty_+ B_{\gl}S p(n)$ in $\GH$.

\begin{cor}\label{cor:split_B_gl}
  For every $n\geq 0$, the morphisms
  \begin{align*}
  \bigvee \Psi_{k,n} \colon  \bigvee_{k=0,\dots,n}  \ \Sigma^\infty M(k,\mR)\
    &\to \    \Sigma_+^\infty B_{\gl}O(n)\   \ ,\\
      \bigvee  \Psi_{k,n}\colon \bigvee_{k=0,\dots,n}  \ \Sigma^\infty M(k,\mC)\
    &\to \      \Sigma_+^\infty B_{\gl}U(n) 
      \text{\qquad and}\\
   \bigvee \Psi_{k,n}\colon  \bigvee_{k=0,\dots,n}  \ \Sigma^\infty M(k,\mH)  \
    &\to \    \Sigma_+^\infty B_{\gl}S p(n)
  \end{align*}
  are isomorphisms in the global stable homotopy category.
\end{cor}
\begin{proof}
  We give the argument in the orthogonal case; the unitary and symplectic cases are analogous.
  The composite homomorphism
  \begin{align*}
    \gh{\Sigma^\infty_+ B_{\gl} O(n),E} \
    &\xra{\gh{\bigvee \Psi_{k,n},E}}\  \gh{\bigvee_{k=0,\dots,n}  \Sigma^\infty M(k,\mR),E}\\
    &\xra{\text{eval at $w_{k,\mR}$}}\  \bigoplus_{k=0,\dots,n} \ker\big(\res^{O(k)}_{O(k-1)}:\pi_0^{O(k)}(E)\to\pi_0^{O(k-1)}(E)\big)\
\xra{\sum \psi_{k,n}}\ \pi_0^{O(n)}(E)
  \end{align*}
  is evaluation at the stable tautological class $e_{O(n)}$, and hence an isomorphism by
  \cite[Theorem 4.4.3]{schwede:global}.
  In this composite, the second map is an isomorphism by
  Corollary \ref{cor:represent kernel}, and the third map is an isomorphism by 
  Theorem \ref{thm:GF splitting}.
  So the map $\gh{\bigvee \Psi_{k,n},E}$ is an isomorphism.
  Because $E$ is an arbitrary object of the global stable homotopy category, this proves the claim.
\end{proof}

If we apply the forgetful functor
\[ U\ : \ \GH \ \to \ \SH \]
from the global stable homotopy category to the non-equivariant stable homotopy category
to Corollary \ref{cor:split_B_gl},
we obtain the stable splittings due to Snaith 
\cite[Theorem 4.2]{snaith:algebraic_cobordism}, \cite[Theorem 2.2]{snaith:localied_stable}
 and Mitchell-Priddy \cite[Theorem 4.1]{mitchell-priddy:double_coset}.
If $G$ is a compact Lie group, we can apply the forgetful functor \cite[Theorem 4.5.23]{schwede:global}
\[ U_G \ : \ \GH \ \to \ G\text{-}\SH \]
from the global stable homotopy category to the genuine $G$-equivariant stable homotopy category.
This forgetful functor turns the splittings of $B_{\gl}O(n)$, $B_{\gl}U(n)$ and $B_{\gl}S p(n)$
of Corollary \ref{cor:split_B_gl}
into $G$-equivariant stable splittings of the classifying $G$-spaces
for $G$-equivariant real, complex and quaternionic vector bundles.
To the best of my knowledge, these $G$-equivariant stable splittings have not been observed before.

\section{Regularity of Euler classes}
\label{sec:regularity}

In this section we apply our splitting results
to derive the regularity of certain equivariant Euler classes 
of the global Thom spectra $\bMU$ and $\bMO$,
see Corollaries \ref{cor:Euler_regular_complex} -- \ref{cor:block_regularity_real}.

As we already mentioned, the equivariant homotopy groups
of a global spectrum (i.e., an object of the global stable homotopy category,
represented by an orthogonal spectrum)
form a graded global functor.
If the global spectrum $E$ is a global homotopy ring spectrum
(i.e., a monoid in the global stable homotopy category under the globally derived smash product),
the equivariant homotopy groups form graded rings, and for all compact Lie groups $G$ and $K$,
the groups $\pi_*^{G\times K}(E)$ are naturally a graded module over the graded ring
$\pi_*^G(E)$, via inflation along the projection $G\times K\to G$.

In the next corollary, we continue to write $\nu_{n,\mR}$,
$\nu_{n,\mC}$ and $\nu_{n,\mH}$ for the tautological representation
of $O(n)$, $U(n)$ and $S p(n)$ on $\mR^n$, $\mC^n$, and $\mH^n$, respectively. 
We write $a_{n,\mR}$ for the Euler class of $\nu_{n,\mR}$, i.e.,
the element of $\pi_0^{O(n)}(\Sigma^\infty S^{\nu_{n,\mR}})$
represented by the fixed point inclusion $S^0\to S^{\nu_{n,\mR}}$,
and similarly for $a_{n,\mC}$ and $a_{n,\mH}$.

\begin{cor}\label{cor:pre-Euler_regular}
  For every orthogonal spectrum $E$, every compact Lie group $G$ and every $n\geq 1$,
  the sequences of graded abelian groups
  \[ 0 \ \to \ \pi_{*+\nu_{n,\mR}}^{G\times O(n)}(E)\ \xra{-\cdot a_{n,\mR}}\
    \pi_*^{G\times O(n)}(E)\ \xra{\res^{G\times O(n)}_{G\times O(n-1)}}\ \pi_*^{G\times O(n-1)}(E) \ \to \ 0
  \]
    \[ 0 \ \to \ \pi_{*+\nu_{n,\mC}}^{G\times U(n)}(E)\ \xra{-\cdot a_{n,\mC}}\
    \pi_*^{G\times U(n)}(E)\ \xra{\res^{G\times U(n)}_{G\times U(n-1)}}\ \pi_*^{G\times U(n-1)}(E) \ \to \ 0
  \]
  and
    \[ 0 \ \to \ \pi_{*+\nu_{n,\mH}}^{G\times Sp(n)}(E)\ \xra{-\cdot a_{n,\mH}}\
    \pi_*^{G\times Sp(n)}(E)\ \xra{\res^{G\times Sp(n)}_{G\times Sp(n-1)}}\ \pi_*^{G\times Sp(n-1)}(E) \ \to \ 0
  \]
  are split exact. 
  If $E$ is a global homotopy ring spectrum, then the splittings
  can be chosen as homomorphisms of graded $\pi_*^G(E)$-modules. 
\end{cor}
\begin{proof}
  As usual, we prove the orthogonal case, and the unitary and symplectic cases are analogous.
  The cofiber sequence of based $O(n)$-spaces
\[ O(n)/O(n-1)_+ \ \to \ S^0 \ \xra{\text{incl}} \ S^{\nu_{n,\mR}} \ \to \ S^1\sm O(n)/O(n-1)_+\]
becomes a cofiber sequence of $(G\times O(n))$-spaces by letting $G$ act trivially.
It induces a long exact sequence in $(G\times O(n))$-equivariant $E$-cohomology
that we can interpret as a long exact sequence of $R O$-graded equivariant homotopy groups:
\[ \dots\ \to \ \pi_{*+1}^{G\times O(n-1)}(E) \ \xra{\ \partial\ } \  \pi_{*+\nu_{n,\mR}}^{G\times O(n)}(E)\ \xra{-\cdot a_{n,\mR}}\
  \pi_*^{G\times O(n)}(E)\ \xra{\res^{G\times O(n)}_{G\times O(n-1)}}\ \pi_*^{G\times O(n-1)}(E) \ \to \ \dots
\]
The restriction homomorphism $\res^{G\times O(n)}_{G\times O(n-1)}$ is split surjective
by Theorem \ref{thm:GF splitting} and Proposition \ref{prop:shifted splitting},
so the long exact sequence decomposes into short exact sequences.
\end{proof}

We let $\bMU$ denote the global Thom ring spectrum
defined in \cite[Example 6.1.53]{schwede:global}.
For every compact Lie group $G$, the underlying $G$-homotopy type of $\bMU$
is that of tom Dieck's homotopical equivariant bordism \cite{tomDieck:bordism_integrality}.
For abelian compact Lie groups, the equivariant cohomology theory represented by $\bMU$
is the universal complex-oriented equivariant cohomology theory
\cite[Theorem 1.2]{cole-greenlees-kriz:universal}.
On the family of all abelian compact Lie groups,
the equivariant homotopy groups of $\bMU$ carry the universal
global group law \cite[Theorem C]{hausmann:group_law}.

Since the global theory $\bMU$ is complex-oriented, every unitary representation
$W$ of a compact Lie group $G$ has an Euler class $e_{G,W}\in\bMU^{2 n}_G$,
where $n=\dim_\mC(W)$;
by definition, $e_{G,W}$ is the restriction of the Thom class
\[ \sigma_{G,W} \in \widetilde{\bMU}^{2 n}_G(S^W) \]
along the inclusion $S^0\to S^W$.

\begin{cor}\label{cor:Euler_regular_complex}
  For every compact Lie group $G$, every character $\chi:G\to U(1)$ and every $n\geq 1$,
 the Euler class of the $(G\times U(n))$-representation $\chi\tensor\nu_{n,\mC}$
  is a non zero-divisor in the graded-commutative ring $\bMU^*_{G\times U(n)}$.
\end{cor}
\begin{proof}
  We start with the special case where $\chi$ is the trivial character.
  Then the representation in question is $p^*(\nu_{n,\mC})$, the restriction
  of the tautological $U(n)$-representation along the projection $G\times U(n)\to U(n)$.
  The equivariant Thom isomorphism identifies the group
  $\pi_{k+p^*(\nu_{n,\mC})}^{G\times U(n)}(\bMU)=\widetilde{\bMU}^0_{G\times U(n)}(S^{k+ p^*(\nu_{n,\mC})})$
  with the group $\bMU^{-k-2 n}_{G\times U(n)}$, in a way that takes multiplication by
  the class $a_{n,\mC}$
  to multiplication by the Euler class of the representation $p^*(\nu_{n,\mC})$.
  Corollary \ref{cor:pre-Euler_regular} thus shows that the Euler class of 
  $p^*(\nu_{n,\mC})$ is a non zero-divisor.

  In the general case, the map
  \[ \psi \ : \ G\times U(n)\ \to \ G\times U(n)\ , \quad \psi(g,A)\ = \ (g,\chi(g)\cdot A) \]
  is an isomorphism of Lie groups,
  and $\chi\tensor\nu_{n,\mC}$ is the restriction of the representation $p^*(\nu_{n,\mC})$ along $\psi$.
  Restriction along $\psi$ is an isomorphism of graded rings
  \[ \psi^*\ : \ \bMU^*_{G\times U(n)} \ \to \ \bMU^*_{G\times U(n)}\]
  that sends the Euler class of $p^*(\nu_{n,\mC})$ to the Euler class of
  $\psi^*(p^*(\nu_{n,\mC}))=\chi\tensor\nu_{n,\mC}$.
  Since the Euler class of $p^*(\nu_{n,\mC})$ is a non zero-divisor by the first part,
  the Euler class of $\chi\tensor\nu_{n,\mC}$ is a non zero-divisor, too.
\end{proof}

  The special case $G=e$ of Corollary \ref{cor:Euler_regular_complex}
  shows that the Euler class of the
  tautological complex $U(n)$-representation $\nu_{n,\mC}$ is a non zero-divisor in the ring
  $\bMU^*_{U(n)}$. The following corollary generalizes this.

\begin{cor}\label{cor:block regularity}
  For all $k_1,\dots,k_m\geq 1$ with $k_1+\dots+ k_m=n$,
  the Euler class of the tautological representation of
  the group $U(k_1)\times\dots\times U(k_m)$ on $\mC^n$
  is a non zero-divisor in the graded ring $\bMU^*_{U(k_1)\times\dots\times U(k_m)}$.
\end{cor}
\begin{proof}
  The tautological representation of $U(k_1)\times\dots\times U(k_m)$
 splits as a direct sum
 \[  p_1^*(\nu_{k_1,\mC})\oplus\dots\oplus p_m^*(\nu_{k_m,\mC})\ , \]
  where $p_i:U(k_1)\times\dots\times U(k_m)\to U(k_i)$ is the projection to the $i$-th factor.
  The Euler class of a direct sum is the product of the Euler classes, so
  \[ e_{U(k_1)\times\dots\times U(k_m),\nu_{n,\mC}}\ = \ p_1^*(\nu_{k_1,\mC})\cdot\ldots\cdot p_m^*(\nu_{k_m,\mC})\ . \]
  Each factor is a non zero-divisor
  by Corollary \ref{cor:Euler_regular_complex}, hence so is the product.
\end{proof}

Corollaries \ref{cor:Euler_regular_complex} and \ref{cor:block regularity}
work more generally for all globally complex-oriented
homotopy-commu\-ta\-tive global homotopy-ring spectra,
i.e., commutative monoids, under derived smash product,
in the global stable homotopy category that come equipped with coherent
and natural Thom isomorphisms for equivariant complex vector bundles.\medskip

We let $\bMO$ denote the global Thom ring spectrum
defined in \cite[Example 6.1.7]{schwede:global}.
For every compact Lie group $G$, the underlying $G$-homotopy type of $\bMO$
is the real analog of tom Dieck's homotopical equivariant bordism \cite{tomDieck:bordism_integrality}.
By a theorem of Br{\"o}cker and Hook \cite[Theorem 4.1]{brocker-hook},
the $G$-equivariant homology theory represented by $\bMO$ is stable equivariant bordism.
Restricted to elementary abelian 2-groups,
the equivariant homotopy groups of $\bMO$ carry the universal
global 2-torsion group law \cite[Theorem D]{hausmann:group_law}.

Since the global theory $\bMO$ is real-oriented, every orthogonal representation $V$
of a compact Lie group $G$ has an Euler class $e_{G,V}\in\bMO^n_G$,
where $n=\dim_\mR(V)$. The analogous arguments as in the complex case in
Corollaries \ref{cor:Euler_regular_complex} and 
\ref{cor:block regularity} prove the following real counterparts.

\begin{cor}\label{cor:Euler_regular_real}
  For every compact Lie group $G$, every continuous homomorphism $\chi:G\to O(1)$ and every $n\geq 1$,
  the Euler class of the $(G\times O(n))$-representation $\chi\tensor\nu_{n,\mR}$
  is a non zero-divisor in the graded-commutative ring $\bMO^*_{G\times O(n)}$.
\end{cor}

\begin{cor}\label{cor:block_regularity_real}
  For all $k_1,\dots,k_m\geq 1$ with $k_1+\dots+ k_m=n$,
  the Euler class of the tautological representation of
  the group $O(k_1)\times\dots\times O(k_m)$ on $\mR^n$
  is a non zero-divisor in the graded ring $\bMO^*_{O(k_1)\times\dots\times O(k_m)}$.
\end{cor}

Corollaries \ref{cor:Euler_regular_real} and \ref{cor:block_regularity_real}
work more generally for all
globally real-oriented homotopy-commutative global homotopy-ring spectra,
i.e., commutative monoids, under derived smash product,
in the global stable homotopy category that come equipped with coherent
and natural Thom isomorphisms for equivariant real vector bundles.

\section{Alternating, special orthogonal and special unitary groups}\label{sec:nonex}

The families of alternating groups, special orthogonal groups and special unitary groups
have the same kind of structure as the symmetric, orthogonal, unitary and symplectic groups;
so one might wonder about the existence of splittings for the values of global functors
at $A_n$, $S O(n)$ and $S U(n)$. In this section we complete the picture by showing
that the restriction homomorphisms between adjacent groups in these
families do {\em not} split naturally, except in some low-dimensional cases
and for half of the special orthogonal groups.

In \cite{henn-mui}, Dold's method is adapted to obtain
non-equivariant stable splittings of the classifying spaces of
alternating, special orthogonal and special unitary groups
after localizing at specific primes.
Since these are {\em not} integral splittings
and they don't involve adjacent groups from the respective family,
the results are coarser than those of Dold, Snaith and Mitchell-Priddy
for $\Sigma_n$, $O(n)$, $U(n)$ and $S p(n)$.

\begin{eg}[Alternating groups]
  The standard embeddings $i_3:e=A_2\to A_3$ and $i_4:A_3\to A_4$ admit unique retractions by group
  homomorphisms; so for every global functor $F$,
  the restriction homomorphisms $F(i_3^*):F(A_3)\to F(A_2)$
  and $F(i_4^*):F(A_4)\to F(A_3)$ are naturally split by the corresponding inflation homomorphisms.
  
  For $n\geq 5$, the restriction homomorphism $F(i_n^*)$ is not in general surjective.
  For $n\geq 5$ and $n\ne 6,8$, the complex representation ring global functor
   \cite[Example 4.2.8 (iv)]{schwede:global} is a witness.
  Indeed, for such $n$, there are non-conjugate elements of $A_{n-1}$
  that become conjugate in $A_n$.
  There is thus a complex representation of $A_{n-1}$ whose character
  takes different values on these elements, and the class of this representation
  cannot be in the image of the restriction homomorphism
  $\res^{A_n}_{A_{n-1}}:R U(A_n)\to R U(A_{n-1})$.

  The two remaining cases can be settled by group cohomology with mod-2 and mod-3 coefficients:
  the map $i_6^*:H^3(A_6,\mF_2)\to H^3(A_5,\mF_2)$ is not surjective,
  and the map  $i_8^*:H^2(A_8,\mF_3)\to H^2(A_7,\mF_3)$ is not surjective.
\end{eg}

\begin{eg}[Special orthogonal groups]
  For the special orthogonal groups,
  Theorem \ref{thm:GF splitting} implies a natural splitting in half of the cases,
  ultimately because the group $O(2 m-1)$ is isomorphic to $\{\pm 1\}\times S O(2 m-1)$.
  Indeed,  the continuous homomorphism $r:O(2m-1)\to S O(2m-1)$ defined by $r(A)=\det(A)\cdot A$
  is a retraction to the inclusion.
  So for every global functor $F$, 
  inflation along $r$ is a natural section to
  restriction from $F(O(2m-1))$ to $F(SO(2m-1))$.
  In combination with Theorem \ref{thm:GF splitting}, this shows that
  restriction from $F(O(2m))$ to $F(SO(2m-1))$ is naturally split;
  hence also the restriction homomorphism
  \[ F(i_{2m}^*) \colon F(S O(2 m))\to F(S O(2 m-1)) \]
  is a naturally split epimorphism.
  Because $F(i_{2m}^*)$ is surjective, the restriction homomorphism
  \[ F(i_{2m}^*\circ i_{2m+1}^*) \colon F(S O(2 m+1))\to F(S O(2 m-1)) \]
  is also surjective, and hence a naturally split epimorphism
  by Proposition \ref{prop:shifted splitting}.
  The group $F(S O(2 m+1))$ then naturally splits as
  \[  F(S O(2 m+1))\ \iso \ \bigoplus_{k=0,\dots, m}
    \ker\left( \res^{S O(2 k+1)}_{SO (2 k-1)}:F(SO(2k+1))\to F(SO(2k-1))\right) \ ,\]
  with the interpretation that the summand for $k=0$ is the value at the trivial group $S O(1)$.
  The underlying non-equivariant stable splitting of $B S O(2m+1)$
  goes back to Snaith \cite[Theorem 4.3]{snaith:algebraic_cobordism}.
  
  In the case of opposite parities, the map
  $B i_{2m+1}:B S O(2 m)\to B S O(2 m +1)$
  induced by the standard embedding is not even stably split
  in the non-equivariant sense.
  So its global analog cannot split in the
  global stable homotopy category, either, and the corresponding restriction homomorphism
  for global functors are not generally surjective, compare Proposition \ref{prop:shifted splitting}.
  To see this, we recall that 
  the mod-2 cohomology of $B S O(k)$ is a polynomial algebra
  on the Stiefel-Whitney classes $w_2,\dots, w_k$,
  and for $k=2m+1$, the relation $Sq^1(w_{2m})=w_{2m+1}$ holds.
  Since $w_{2m+1}$ is in the kernel of the restriction homomorphism
  \[ \res^{S O(2m +1)}_{S O(2m)}\ : \ H^*(B S O(2m+1),\mF_2)\ \to \ H^*(B S O(2 m),\mF_2) \ , \]
  but $w_{2m}$ restricts non-trivially,
  this restriction does not admit a section that is linear over the mod-2 Steenrod algebra.
\end{eg}

\begin{eg}[Special unitary groups]
  As the group $S U(1)$ is trivial, the standard embedding $i_2:S U(1)\to S U(2)$
  has a retraction by a continuous homomorphism, and for every global functor $F$,
  the corresponding restriction homomorphism $F(i_3^*):F(S U(2))\to F(S U(1))$
  is a naturally split epimorphism.
  
  For $n\geq 3$, restriction homomorphism $F(i_n^*):F(S U(n))\to F(S U(n-1))$
  is not in general surjective. For $n=3$, the relation
  $Sq^2(\bar c_2)=\bar c_3$ in  $H^*(B S U(3),\mF_2)$
  shows that the map $B i_3:B S U(2)\to B S U(3)$
  is not even stably split in the non-equivariant sense;
  here $\bar c_k$ is the mod-2 reduction of the $k$-th Chern class.
  For $n\geq 4$, the Burnside ring global functor $\mA=\bA(e,-)$
  \cite[Example 4.2.8 (i)]{schwede:global}
  is a witness that there is no natural splitting, i.e., the map
  \[ \res^{S U(n)}_{S U(n-1)} \ : \ \mA(S U(n))\ \to \ \mA(S U(n-1))\]
  is not surjective.
  A specific element that is not in the image can be obtained as follows.
  The {\em reduced natural representation} of the alternating group $A_n$
  is the $(n-1)$-dimensional complex vector space
  \[\{ (x_1,\dots,x_n)\in\mC^n \ : \ x_1+\dots+x_n=0\} \]
  with $A_n$-action by permutation of coordinates.
  This action is faithful and by isometries of determinant 1.
  A choice of orthonormal basis identifies $A_n$ with a subgroup of $S U(n-1)$,
  well-defined up to conjugacy. For $n\geq 4$, the Weyl group of $A_n$ in $S U(n-1)$
  is finite, and the transfer $\tr_{A_n}^{S U(n-1)}(1)$ is an element of infinite order
  in $\mA(S U(n-1))$ that is not the restriction of any class in $\mA(S U(n))$.
\end{eg}

\end{document}